\documentclass[12pt]{article}
\usepackage{times}
\usepackage{amsfonts, amsmath, amssymb, amsthm, mathtools, dsfont}
\usepackage[colorlinks=true,citecolor=blue]{hyperref}
\usepackage{setspace, enumerate}
\theoremstyle{plain}
\newtheorem{theorem}{Theorem}[section]
\newtheorem{proposition}[theorem]{Proposition}
\newtheorem{lemma}[theorem]{Lemma}
\newtheorem{observation}[theorem]{Observation}
\newtheorem{corollary}[theorem]{Corollary}
\newtheorem{claim}[theorem]{Claim}
\theoremstyle{definition}
\newtheorem{example}[theorem]{Example}
\newtheorem{remark}[theorem]{Remark}
\newtheorem*{rmk}{Remark}

\DeclareMathOperator{\st}{st}
\DeclareMathOperator{\nc}{NC}
\DeclareMathOperator{\ii}{I}
\DeclareMathOperator{\starcluster}{SC}

\title{Domination numbers and noncover complexes of hypergraphs}
\author{
Jinha Kim\thanks{Discrete Mathematics Group, Institute for Basic Science (IBS), Daejeon, Korea.}
~\thanks{\texttt{jinhakim@ibs.re.kr}}
\and
Minki Kim\footnotemark[1]
~\thanks{Corresponding author. \texttt{minkikim@ibs.re.kr}}}
\date\today

\begin{document}
\maketitle
\begin{abstract}
Let $\mathcal{H}$ be a hypergraph on a finite set $V$.
A {\em cover} of $\mathcal{H}$ is a set of vertices that meets all edges of $\mathcal{H}$.
If $W$ is not a cover of $\mathcal{H}$, then $W$ is said to be a {\em noncover} of $\mathcal{H}$.
The {\em noncover complex} of $\mathcal{H}$ is the abstract simplicial complex whose faces are the noncovers of $\mathcal{H}$.
In this paper, we study homological properties of noncover complexes of hypergraphs.
In particular, we obtain an upper bound on their Leray numbers.
The bound is in terms of hypergraph domination numbers.
Also, our proof idea is applied to compute the homotopy type of the noncover complexes of certain uniform hypergraphs, called {\em tight paths} and {\em tight cycles}.
This extends to hypergraphs known results on graphs.
\end{abstract}

\section{Introduction}
A {\em hypergraph} on $V$ is a collection of non-empty subsets of $V$, called {\em edges}.
When $\mathcal{H}$ is a hypergraph on $V = V(\mathcal{H})$, we call the set $V$ the {\em vertex set} of $\mathcal{H}$.
Note that if $V = \emptyset$ then $\mathcal{H} = \emptyset$.
A singleton edge $\{v\}\in\mathcal{H}$ is called a {\em loop}.
For a positive integer $k$, a hypergraph is said to be {\em $k$-uniform} if every edge has size $k$.
For example, graphs are $2$-uniform hypergraphs.
Throughout this paper, we assume every hypergraph has a finite vertex set and no two edges in a hypergraph are identical.

Let $\mathcal{H}$ be a hypergraph on $V$.
A subset $W$ of $V$ is said to be {\em independent} if it contains no edge of $\mathcal{H}$.
An {\em abstract simplicial complex} on $V$ is a family of subsets of $V$ that is closed under the operation of taking subsets.
The set $\ii(\mathcal{H})$ of independent sets of $\mathcal{H}$ is clearly an abstract simplicial complex.
It is called the {\em independence complex} of $\mathcal{H}$.

A {\em cover} of $\mathcal{H}$ is a vertex subset $W$ of $V$ that meets all edges of $\mathcal{H}$.
Clearly, $W$ is a cover if and only if $V \setminus W$ is an independent set.
$W$ is said to be a {\em noncover} of $\mathcal{H}$ if it is not a cover of $\mathcal{H}$.
Let $\nc(\mathcal{H})$ be the complex of noncovers of $\mathcal{H}$.
Note that every maximal face of $\nc(\mathcal{H})$ is the complement of an edge of $\mathcal{H}$.

For a simplicial complex $K$, the {\em (combinatorial) Alexander dual} of $K$ is the simplicial complex \[D(K) := \{\sigma \subset V: V\setminus \sigma\notin K\}.\]
Observe that the noncover complex of a hypergraph $\mathcal{H}$ is the Alexander dual of the independence complex of $\mathcal{H}$.
Let $\tilde{H}_i(K)$ be the $i$-dimensional reduced homology group of $K$.
In this paper, the coefficients of homology groups are taken in $\mathbb{Z}_2$.
	The homology groups of a simplicial complex $K$ and those of its dual $D(K)$ are related by a duality theorem~\cite{Kal83, Sta82}.
\begin{theorem}[The duality theorem]\label{duality theorem}
Let $K$ be a simplicial complex on $V$.
Then $\tilde{H}_i(D(K)) \cong \tilde{H}_{|V|-i-3}(K)$ for all $i$.
\end{theorem}
See also \cite{BM09} for a short proof.

Here are three examples.
Let $\tilde{\beta}_i(\cdot)$ denote the $i$-th reduced Betti number of the complex, i.e. the rank of the $i$-dimensional reduced homology group.
\begin{itemize}
	\item If $\mathcal{H}$ has no edge, then $\ii(\mathcal{H})$ is the {\em simplex} on $V$, i.e. $\ii(\mathcal{H}) = 2^V$, and $\nc(\mathcal{H})$ is a {\em void complex}, i.e. $\nc(\mathcal{H}) = \emptyset$. 
In this case, both complexes have vanishing reduced homology in every dimension.
	\item Let $\mathcal{H} = \{V\}$.
	Then $\ii(\mathcal{H})$ is the boundary of the simplex on $V$, $\tilde{\beta}_{|V|-2}(\ii(\mathcal{H})) = 1$ and $\tilde{\beta}_{i}(\ii(\mathcal{H})) = 0$ for every $i \neq |V|-2$.
	On the other hand, $\nc(\mathcal{H})$ is an {\em empty complex}, i.e. $\nc(\mathcal{H}) = \{\emptyset\}$.
	Note that if $K$ is an empty complex then $\tilde{\beta}_{-1}(K) = 1$ and $\tilde{\beta}_{i}(K) = 0$ for every $i \neq -1$.	
	\item Suppose $\mathcal{H}$ contains a loop of $v$ for each $v \in V$, i.e. $\binom{V}{1} \subseteq \mathcal{H}$. Then $\ii(\mathcal{H})$ is an empty complex and $\nc(\mathcal{H})$ is the boundary of the simplex on $V$.
\end{itemize}

In this paper, we study relations between domination numbers for hypergraphs and homological properties of noncover complexes and independence complexes of hypergraphs.

\subsection{Lerayness and connectivity of simplicial complexes}
	A simplicial complex $K$ on $V$ is said to be {\em $d$-Leray} if $\tilde{H}_i(K[W]) = 0$ for all $i \geq d$ and $W \subset V$, where $K[W]$ is the subcomplex of $K$ induced on $W$.
	The {\em Leray number} $L(K)$ of $K$ is the minimum integer $d$ such that $K$ is $d$-Leray.
	For example, the boundary of an $n$-simplex is $n$-Leray.
	
	A closely related parameter is the (homological) connectivity.
	A simplicial complex $K$ on $V$ is said to be {\em (homologically) $k$-connected} if $\tilde{H}_i(K) = 0$ for all $-1 \leq i \leq k$.
	Let  $\eta(K)$ be the maximal integer $k$ for which $K$ is $(k-2)$-connected.
	If there is no such $k$ (in particular, if $K$ is contractible), then we write $\eta(K) = \infty$.
	For example, any non-empty complex $K$ has $\eta(K) \geq 1$ and the boundary of an $n$-simplex has $\eta(\partial\Delta_n) = n$.
    Theorem~\ref{duality theorem} implies that any simplicial complex $K$ has $L(K) \leq d$ if $\eta(D(K[W]) \geq |W| - d - 1$ for every $W \subset V$.

\subsection{Domination numbers for hypergraphs}
    We will define three domination parameters of hypergraphs.
    
	Let $\mathcal{H}$ be a hypergraph on $V$.
	Let $v$ be a vertex in $V$ and $W$ be a subset of $V$.
	Here are three different definitions of ``domination'' in hypergraphs.
	\begin{itemize}
	\item $W$ {\em totally dominates} $v$ if $v$ is a loop in $\mathcal{H}$ or there exists a vertex $w \neq v$ in $W$ such that $v$ and $w$ belong to the same edge. This notion appears in a paper by Bujt\'{a}s et al. \cite{BHTY14}.
	\item $W$ {\em (weakly) dominates} $v$ if $W$ totally dominates $v$ or $W$ contains $v$. This notion was introduced by Acharya~\cite{Ach07}.
	\item $W$ {\em strongly totally dominates} $v$ if there exists $W' \subset W \setminus \{v\}$ such that $W' \cup \{v\}$ is an edge of $\mathcal{H}$. In particular, the empty set dominates $v$ if $v$ is a loop.
	\end{itemize}
	Note that if $W$ strongly totally dominates $v$ then it dominates $v$ in any of the other notions of domination.
	In this paper, we say $W$ {\em dominates} $v$ if $W$ strongly totally dominates $v$.

	Let $A$ be a subset of $V$.
	If $W \subset V$ dominates every vertex in $A$, then we say $W$ {\em dominates} $A$.	
	The {\em strong total domination number of $A$ in $\mathcal{H}$} is the integer
\[\gamma(\mathcal{H}; A) := \min\{|W|:W\subset V, ~W\text{ dominates }A\}.\]
	The {\em strong total domination number} $\tilde{\gamma}(\mathcal{H})$ of $\mathcal{H}$ is the strong total domination number of the whole vertex set, i.e. $\tilde{\gamma}(\mathcal{H}) = \gamma(\mathcal{H}; V)$.
	
	$A \subset V$ is said to be {\em strongly independent} in $\mathcal{H}$ if it is independent and every edge of $\mathcal{H}$ contains at most one vertex of $A$.	The {\em strong independence domination number} of $\mathcal{H}$ is the integer
\[\gamma_{si}(\mathcal{H}) := \max\{\gamma(\mathcal{H}; A):A\text{ is a strongly independent set of }\mathcal{H}\}.\]

	The {\em edgewise-domination number} of $\mathcal{H}$ is the minimum number of edges whose union dominates the whole vertex set $V$, i.e.
	\[\gamma_E(\mathcal{H}) := \min\{|\mathcal{F}|: \mathcal{F} \subset \mathcal{H}, \bigcup_{F \in \mathcal{F}}F\text{ dominates }V\}.\]
	Clearly, if $\mathcal{H}$ is $k$-uniform, then $\gamma_E(\mathcal{H}) \geq \left\lceil\frac{\tilde{\gamma}(\mathcal{H})}{k}\right\rceil$.
	
	Note that if $\binom{V}{1} \subset \mathcal{H}$, then $\tilde{\gamma}(\mathcal{H}) = \gamma_{si}(\mathcal{H}) = \gamma_E(\mathcal{H}) = 0$.
    If $\mathcal{H}$ has an {\em isolated vertex} $v$, i.e. if no edge of $\mathcal{H}$ contains $v$, then there does not exist $W\subset V$ that dominates $v$.
    In this case $\tilde{\gamma}(\mathcal{H})$, $\gamma_{si}(\mathcal{H})$ and $\gamma_E(\mathcal{H})$ are defined as $\tilde{\gamma}(\mathcal{H}),  \gamma_{si}(\mathcal{H}), \gamma_E(\mathcal{H})= \infty$.
	
	\smallskip

    Bounding $\eta(\ii(\mathcal{H}))$ in terms of the above domination parameters when $\mathcal{H}$ is a ($2$-uniform) graph has been studied in many papers.
    The following theorem summarizes such results in \cite{AH00, ACK02, Chud}. (See also \cite{Mes01, Mes03}.)
	\begin{theorem}\label{eta graph}
	For every graph $G$, $\eta(\ii(G)) \geq \max\{\left\lceil\frac{\tilde{\gamma}(G)}{2}\right\rceil, \gamma_{si}(G), \gamma_E(G)\}$.
	\end{theorem}	
	
	An immediate corollary of Theorem~\ref{duality theorem} and Theorem~\ref{eta graph} is:
    \begin{corollary}\label{eta dual graph}
	For every graph $G$, $\tilde{H}_i({\nc(G)}) = 0$ for all $i \geq |V(G)| - \max\{ \left\lceil\frac{\tilde{\gamma}(G)}{2}\right\rceil, \gamma_{si}(G), \gamma_E(G)\} - 1$.
	\end{corollary}

\subsection{Main results}
	The first part of this paper establishes a hypergraph analogue of Corollary~\ref{eta dual graph}.
	\begin{theorem}\label{eta dual}
	For every hypergraph, $\tilde{H}_i({\nc(\mathcal{H})}) = 0$ for all $i \geq |V(\mathcal{H})| - \max\{ \left\lceil\frac{\tilde{\gamma}(\mathcal{H})}{2}\right\rceil, \gamma_{si}(\mathcal{H}), \gamma_E(\mathcal{H})\} - 1$.
	\end{theorem}
    By Theorem~\ref{duality theorem}, Theorem~\ref{eta dual} implies the hypergraph analogue of Theorem~\ref{eta graph}, giving a lower bound on $\eta(\ii(\mathcal{H}))$.
\begin{corollary}\label{eta}
	$\eta(\ii(\mathcal{H})) \geq \max\{\left\lceil\frac{\tilde{\gamma}(\mathcal{H})}{2}\right\rceil, \gamma_{si}(\mathcal{H}), \gamma_E(\mathcal{H})\}$ for every hypergraph $\mathcal{H}$.
\end{corollary}    

    The proof of Theorem~\ref{eta dual} will be given in Section~\ref{proof main 1}.
    In Section~\ref{genpos}, we present an alternative proof of the main result in \cite{HMM16} as an application of Corollary~\ref{eta}.
    Our proof method also can be applied to generalize \cite[Claim~3.3]{Mes03} which computes the reduced Betti numbers of the independence complexes of paths and cycles.
    In Section~\ref{tight hypergraphs}, we compute the  homotopy types as well as the reduced Betti numbers of the noncover complexes of certain uniform hypergraphs, called {\em tight paths} and {\em tight cycles}.

    \smallskip

	The second part of this paper strengthens Theorem~\ref{eta dual} for some cases.
	We prove upper bounds of $L(\nc(\mathcal{H}))$ in terms of $\tilde{\gamma}(\mathcal{H})$, $\gamma_{si}(\mathcal{H})$, and $\gamma_E(\mathcal{H})$.
	\begin{theorem}\label{leray numbers}
	Let $\mathcal{H}$ be a hypergraph with no isolated vertices.
	Then 
	\begin{enumerate}[(a)]
	\item If $|e| \leq 3$ for every $e \in \mathcal{H}$, then $L(\nc(\mathcal{H})) \leq |V(\mathcal{H})| - \left\lceil\frac{\tilde{\gamma}(\mathcal{H})}{2}\right\rceil -1$.
	\item If $|e| \leq 2$ for every $e \in \mathcal{H}$, then $L(\nc(\mathcal{H})) \leq |V(\mathcal{H})| - \gamma_{si}(\mathcal{H}) -1$.
	\item $L(\nc(\mathcal{H})) \leq |V(\mathcal{H})| - \gamma_{E}(\mathcal{H}) -1$.
	\end{enumerate}
	\end{theorem}
	Note that if a hypergraph $\mathcal{H} \neq \emptyset$ contains an isolated vertex $v$, then the noncover complex $\nc(\mathcal{H})$ is a cone with apex $v$, so it is contractible.
	Hence $L(\nc(\mathcal{H})) = L(\nc(\mathcal{H}'))$, where $\mathcal{H}'$ is the hypergraph obtained from $\mathcal{H}$ by removing all isolated vertices.
	
    The restrictions on the size of edges in (a) and (b) are necessary: see Example~\ref{gamma tilde example} for (a) and Example~\ref{gammasi example} for (b).
    In Section~\ref{sec:rainbow cover}, we present an application of Theorem~\ref{leray numbers} to a ``rainbow cover theorem''.

\begin{rmk}
An extended abstract of this paper can be found in the Proceedings of the 32nd Conference on Formal Power Series and Algebraic Combinatorics (Online) \cite{fpsac}.
\end{rmk}

\medskip
\section{Proof of Theorem~\ref{eta dual}}\label{proof main 1}
In this section, we give a proof of Theorem~\ref{eta dual}.
We start with defining an edge operation on hypergraphs.

\subsection{Edge annihilation}
Let $\mathcal{H}$ be a hypergraph on $V$ and $e \in \mathcal{H}$.
Throughout this paper, we denote the subhypergraph $\mathcal{H} \setminus \{e\}$ by $\mathcal{H} - e$.
Also, we define the {\em edge-annihilation} of $e$ in $\mathcal{H}$ as an operation of obtaining the hypergraph
\[\mathcal{H} \neg ~e := \{ f \setminus e : f \in \mathcal{H}\;\text{and}\;f\nsubseteq e\}\]
from $\mathcal{H}$.
Note that $\mathcal{H} - e$ is a hypergraph on $V$ and $\mathcal{H} \neg ~e$ is a hypergraph on $V \setminus e$.
See Figure~\ref{fig:edge-annihilation} for the illustration of an edge-annihilation.
	\begin{figure}[htbp]
    \centering
    \includegraphics[scale=0.6]{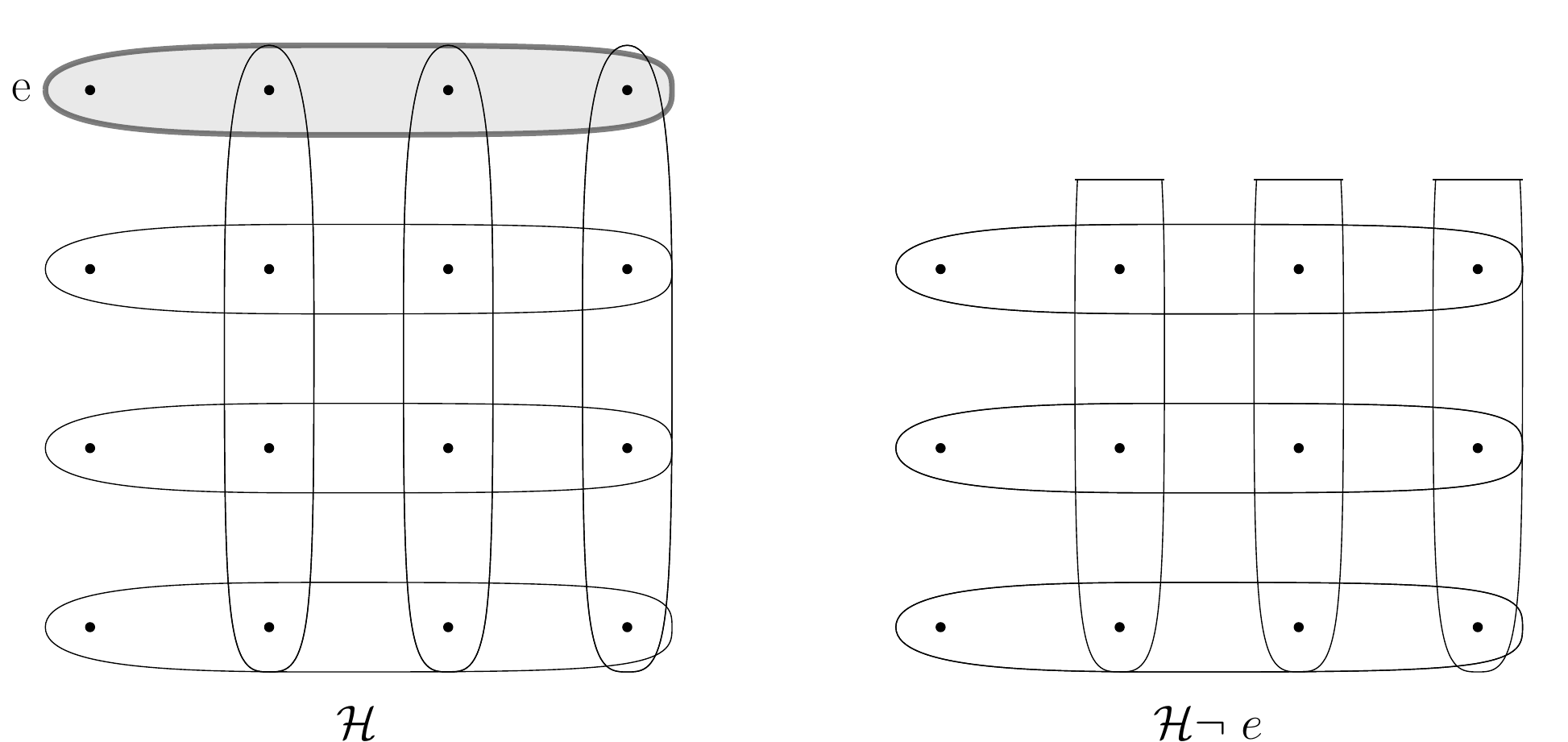}
    \caption{$\mathcal{H}\neg~e$ is obtained from $\mathcal{H}$ by annihilating the edge $e$.}
    \label{fig:edge-annihilation}
\end{figure}

We will show some relations between the domination parameters of $\mathcal{H}$ and those of $\mathcal{H} \neg~e$.
This is a hypergraph analogue of Meshulam's observation for the case of graphs \cite[Theorem 1.2]{Me s03}.

\begin{lemma}\label{inequalities}
Let $\mathcal{H}$ be a hypergraph with vertex set $V \neq \emptyset$.
If $\mathcal{H}$ has no isolated vertices, then each of the following holds:
\begin{enumerate}[(i)]
\item $\tilde{\gamma}(\mathcal{H} \neg~e) \geq \tilde{\gamma}(\mathcal{H}) - 2|e| + 2$ for every $e \in \mathcal{H}$ with $|e| \geq 2$.

\smallskip

\item If $\binom{V}{1} \nsubseteq \mathcal{H}$, then there exists $e \in \mathcal{H}$ such that $|e| \geq 2$ and $\gamma_{si}(\mathcal{H} \neg ~e) \geq \gamma_{si}(\mathcal{H}) - |e| + 1$.

\smallskip

\item $\gamma_E(\mathcal{H}\neg~e) \geq \gamma_E(\mathcal{H})-|e|+1$ for every $e \in \mathcal{H}$ with $|e| \geq 2$.
\end{enumerate}
	\begin{proof}
\begin{enumerate}[(i)]
\item Take any edge $e \in \mathcal{H}$ with $|e| \geq 2$, and let $S \subset V(\mathcal{H} \neg~e)$ be a minimum set which dominates $V\setminus e$ in $\mathcal{H}\neg~e$.
	Then $S \cup e$ dominates $V$ in $\mathcal{H}$, thus we have
	\[\tilde{\gamma}(\mathcal{H} \neg~ e) + |e| = |S \cup e| \geq \tilde{\gamma}(\mathcal{H}).\]
	Since $|e| \geq 2$, we obtain $\tilde{\gamma}(\mathcal{H} \neg~ e) \geq \tilde{\gamma}(\mathcal{H}) - |e| \geq \tilde{\gamma}(\mathcal{H}) - 2|e| + 2$.
	
	\smallskip
	
\item Let $A$ be a strongly independent set of $\mathcal{H}$ such that $\gamma_{si}(\mathcal{H}) = \gamma(\mathcal{H}; A)$.
    Clearly, $A \neq \emptyset$ since $\binom{V}{1} \nsubseteq \mathcal{H}$.
	Pick $v \in A$.
	Since $\mathcal{H}$ contains no isolated vertices in $\mathcal{H}$, there exists $e \in \mathcal{H}$ that contains $v$.
	Note that since $A$ is an independent set of $\mathcal{H}$, $\{v\}$ cannot be an edge in $\mathcal{H}$, thus we have $|e| \geq 2$.
	Let $A'$ be the set of all vertices in $A$ that is not dominated by $e \setminus \{v\}$.
	Then $A'$ is a strongly independent set of $\mathcal{H} \neg~e$.
	Let $S \subset V \setminus e$ be a minimum set which dominates $A'$ in $\mathcal{H} \neg~ e$. 
	Then $S \cup (e\setminus\{v\})$ dominates $A$ in $\mathcal{H}$, thus we have
	\[ \gamma_{si}(\mathcal{H} \neg~ e) + (|e| - 1) \geq |S \cup (e\setminus\{v\})| \geq \gamma_{si}(\mathcal{H}).\]

	\item Let $e$ be an edge in $\mathcal{H}$ with $|e| \geq 2$, and let $\mathcal{F}$ be a subgraph of $\mathcal{H} \neg~ e$ such that $\bigcup_{A \in \mathcal{F}}A$ dominates $V\setminus e$ in $\mathcal{H} \neg~ e$ and $|\mathcal{F}| =\gamma_E(\mathcal{H} \neg~ e)$.
	For each $A \in \mathcal{F}$, there exists $e_A \subset e$ (possibly empty) such that $A \cup e_A \in \mathcal{H}$.
	Since $e \cup \left(\bigcup_{A\in\mathcal{F}} (A \cup e_A)\right)$ dominates $V$ in $\mathcal{H}$, we have
		\[\gamma_E(\mathcal{H} \neg~ e) + (|e| - 1) \geq \gamma_E(\mathcal{H} \neg~ e) + 1 \geq \gamma_E(\mathcal{H}).\]
\end{enumerate}
	\end{proof}
\end{lemma}

Lemma~\ref{inequalities} gives us the following corollary.
\begin{corollary}\label{cor:ineq}
Let $\mathcal{H}$ be a hypergraph with vertex set $V \neq \emptyset$. Let \[g(\mathcal{H}) = \max\{ \left\lceil\frac{\tilde{\gamma}(\mathcal{H})}{2}\right\rceil, \gamma_{si}(\mathcal{H}), \gamma_E(\mathcal{H})\}.\]
For every $e \in \mathcal{H}$, $g(\mathcal{H} - e) \geq g(\mathcal{H})$.
Moreover, if $\mathcal{H}$ has no isolated vertices and $\binom{V}{1} \nsubseteq \mathcal{H}$, then there exists $e \in \mathcal{H}$ such that $g(\mathcal{H} \neg~e) \geq g(\mathcal{H}) - |e| + 1$.
\end{corollary}

\subsection{Noncover complexes for hypergraphs}
When we compute the homology of noncover complexes of hypergraphs, we may assume that every edge is inclusion-minimal.
This observation will be used many times throughout the paper.
\begin{observation}\label{noncover inclusion-minimal}
Let $\mathcal{H}$ be a hypergraph on $V$ and suppose there are two edges $e,f \in \mathcal{H}$ such that $f \subset e$.
Since $V \setminus e \subset V \setminus f$, deleting $e$ from $\mathcal{H}$ does not affect the noncover complex, i.e. $\nc(\mathcal{H}) = \nc(\mathcal{H} - e)$.
\end{observation}

The proof of Theorem~\ref{eta dual} is based on the Mayer-Vietoris exact sequence.
Let $K$ be an abstract simplicial complex and let $A$ and $B$ be complexes such that $K = A \cup B$.
Then the following sequence is exact:
\[\cdots \to \tilde{H}_i (A \cap B) \to \tilde{H}_i(A) \oplus \tilde{H}_i(B) \to \tilde{H}_i(K) \to \tilde{H}_{i-1} (A \cap B) \to \cdots \]
In particular, for any integer $i_0$, if $\tilde{H}_i(A) = \tilde{H}_i(B) = \tilde{H}_{i-1}(A \cap B) = 0$ for all $i \geq i_0$ then $\tilde{H}_i(K) = 0$ for all $i \geq i_0$.

Let $\mathcal{H}$ be a hypergraph on $V$.
For $X \subset V$, let $\Delta_X$ be the simplex on $X$.
\begin{lemma}\label{noncov union}
For every edge $e$ in $\mathcal{H}$, we have $\nc(\mathcal{H}) = \nc(\mathcal{H} - e) \cup \Delta_{V \setminus e}$.
\end{lemma}
\begin{proof}
It is obvious from the definition $\nc(\mathcal{H}) = \bigcup_{e \in \mathcal{H}} \Delta_{V \setminus e}$.
\end{proof}

\begin{lemma}\label{noncov intersection}
For any inclusion-minimal edge $e \in \mathcal{H}$, we have 
\[\nc(\mathcal{H} - e) \cap \Delta_{V \setminus e} = \nc(\mathcal{H} \neg ~e).\]
\end{lemma}
\begin{proof}
It is clear that $\nc(\mathcal{H}\neg~e) \subset \Delta_{V \setminus e}$.
If $\sigma \in \nc(\mathcal{H}\neg~e)$, then there exists an edge $f \nsubseteq e$ in $\mathcal{H}$ such that $f \setminus e \in \mathcal{H}\neg~e$ and $f \setminus e \subset V(\mathcal{H}\neg~e)\setminus \sigma$.
Since $V(\mathcal{H}\neg~e) = V \setminus e$, we have $f \subset V \setminus \sigma$, which implies $\sigma \in \nc(\mathcal{H} - e)$.

For the opposite direction, let $\sigma \in \nc(\mathcal{H}-e) \cap \Delta_{V \setminus e}$.
Since $\sigma \in \nc(\mathcal{H} - e)$, there exists an edge $f \in \mathcal{H} - e$ such that $f \subset V \setminus \sigma$.
It is clear that $f \setminus e \in \mathcal{H}\neg~e$ because $e$ is inclusion-minimal.
Since $\sigma \in \Delta_{V \setminus e}$, we have $f\setminus e \subset V(\mathcal{H}\neg~e) \setminus \sigma$.
It follows that $\sigma \in \nc(\mathcal{H}\neg~e)$.
\end{proof}

Here we assume that a simplex on an empty set is an empty complex.
For each inclusion-minimal edge $e \in \mathcal{H}$, we obtain the following exact sequence by Lemma~\ref{noncov union} and Lemma~\ref{noncov intersection}:
\begin{align}\label{exact noncover}
\begin{split}
&\cdots \to \tilde{H}_{i}(\nc(\mathcal{H} \neg~e)) \to \tilde{H}_i(\nc(\mathcal{H} -e)) \to \tilde{H}_i(\nc(\mathcal{H})) \\
& \to  \tilde{H}_{i-1}(\nc(\mathcal{H} \neg~e)) \to \tilde{H}_{i-1}(\nc(\mathcal{H} -e)) \to \tilde{H}_{i-1}(\nc(\mathcal{H})) \to \cdots.
\end{split}
\end{align}

	\begingroup
	\def\thetheorem{\ref{eta dual}}
	\begin{theorem}
	For every hypergraph $\mathcal{H}$, $\tilde{H}_i({\nc(\mathcal{H})}) = 0$ for all $i \geq |V(\mathcal{H})| - \max\{ \left\lceil\frac{\tilde{\gamma}(\mathcal{H})}{2}\right\rceil, \gamma_{si}(\mathcal{H}), \gamma_E(\mathcal{H})\} - 1$.
	\end{theorem}
	\addtocounter{theorem}{-1}
	\endgroup		
\begin{proof}
If $\mathcal{H} = \emptyset$, then $\nc(\mathcal{H})$ is a void complex.
If $\mathcal{H} \neq \emptyset$ and $v$ is an isolated vertex in $\mathcal{H}$, then $\nc(\mathcal{H})$ is a cone with apex $v$, thus is contractible.
In both cases, we have $\tilde{H}_i({\nc(\mathcal{H})}) = 0$ for all $i$.
Thus we may assume that $\mathcal{H} \neq \emptyset$ and $\mathcal{H}$ contains no isolated vertices.
Note that if $\mathcal{H} \neq \emptyset$ then $V(\mathcal{H}) \neq \emptyset$.

We may further assume $\binom{V(\mathcal{H})}{1} \nsubseteq \mathcal{H}$.
If $\binom{V(\mathcal{H})}{1} \subset \mathcal{H}$, then $\nc(\mathcal{H})$ is the boundary of the simplex on $V(\mathcal{H})$ which has non-vanishing homology only in dimension $|V(\mathcal{H})|-2$.
Thus the statement is true since every vertex can be dominated by an empty set in our situation.

Now we apply induction on $|V(\mathcal{H})| + |\mathcal{H}|$.
As before, let \[g(\mathcal{H}) = \max\{ \left\lceil\frac{\tilde{\gamma}(\mathcal{H})}{2}\right\rceil, \gamma_{si}(\mathcal{H}), \gamma_E(\mathcal{H})\}.\]
Suppose $\mathcal{H}$ contains an edge $e$ that is not inclusion-minimal.
If $\mathcal{H} - e$ has an isolated vertex $v$, then $\nc(\mathcal{H} - e)$ is contractible.
Otherwise, since $g(\mathcal{H}) \leq g(\mathcal{H} - e) < \infty$ and $|V(\mathcal{H} - e)| = |V(\mathcal{H})|$, we have $|V(\mathcal{H}-e)| - g(\mathcal{H} - e) - 1 \leq V(\mathcal{H}) - g(\mathcal{H}) - 1$.
By the induction hypothesis, it follows that $\tilde{H}_i(\nc(\mathcal{H} - e)) = 0$ for all $i \geq |V(\mathcal{H})| - g(\mathcal{H}) - 1$.
Thus the statement is true since $\nc(\mathcal{H}) = \nc(\mathcal{H}-e)$ by Observation~\ref{noncover inclusion-minimal}.

Now we may assume that every edge in $\mathcal{H}$ is inclusion-minimal.
By Corollary~\ref{cor:ineq}, there exists $e \in \mathcal{H}$ such that $g(\mathcal{H} -e) \geq g(\mathcal{H})$ and $g(\mathcal{H} \neg~e) \geq g(\mathcal{H}) - |e| +1$.
Since $|V(\mathcal{H})| - g(\mathcal{H})-1 \geq |V(\mathcal{H}-e)| - g(\mathcal{H}-e)-1$ and $|V(\mathcal{H})| - g(\mathcal{H}) - 2 \geq |V(\mathcal{H}\neg e)| - g(\mathcal{H} \neg e) - 1$, it follows from the induction hypothesis that 
\begin{itemize}
\item $\tilde{H}_i(\nc(\mathcal{H} - e)) = 0$ for all $i \geq |V(\mathcal{H})| - g(\mathcal{H})-1$, and 

\item $\tilde{H}_i(\nc(\mathcal{H} \neg~e)) = 0$ for all $i \geq |V(\mathcal{H})| - g(\mathcal{H}) - 2$.
\end{itemize}
Then an immediate application of the exact sequence \eqref{exact noncover} gives us $\tilde{H}_i({\nc(\mathcal{H})}) = 0$ for all $i \geq |V(\mathcal{H})| - g(\mathcal{H}) - 1$, as required.
\end{proof}

We conclude this section with two remarks.
\begin{remark}\label{remark:star cluster}
Let $X$ be a simplicial complex on $V \neq \emptyset$.
\begin{itemize}
\item For a non-negative integer $k$, the {\em $k$-skeleton} of $X$ is the subcomplex \[X^{(k)} := \{\sigma \in X: |\sigma| \leq k+1\}.\]

\item For $A \subset V$, the {\em star} of $A$ in $X$ is the subcomplex \[\st(X, A) := \{\sigma \in X: A \cup \sigma \in X\},\]
and the {\em star cluster} of $A$ in $X$ is the subcomplex \[\starcluster(X, A) := \bigcup_{v\in A}\st(X, \{v\}).\]
\end{itemize}
The notion of the star cluster was introduced in \cite{Bar13}.
One can show $\eta(\ii(\mathcal{H})) \geq \gamma_{si}(\mathcal{H})$ using the star cluster.

Let $\mathcal{H}$ be a hypergraph on $V$ and $A \subset V$ a subset.
It is straightforward to see that $\sigma \in \ii(\mathcal{H})$ dominates $A$ if and only if $\sigma \notin \starcluster(\ii(\mathcal{H}), A)$.
Therefore, if $A$ is a strongly independent set of $\mathcal{H}$ such that $s = \gamma_{si}(\mathcal{H}) = \gamma(\mathcal{H} ; A)$, then we have $(\ii(\mathcal{H}))^{(s-2)} \subset \starcluster(\ii(\mathcal{H}),A) \subset \ii(\mathcal{H})$ since $|\sigma| \geq s$ whenever $\sigma \notin \starcluster(\ii(\mathcal{H}),A)$.
It follows that $(\ii(\mathcal{H}))^{(s-2)} = (\starcluster(\ii(\mathcal{H}),A))^{(s-2)}$, thus it is sufficient to show that $\starcluster(\ii(\mathcal{H}), A)$ is contractible.
When $\mathcal{H}$ is a graph, it is implied by \cite[Lemma 3.2]{Bar13}.
In general, it follows from the fact that if $A$ is strongly independent, then for every $A' \subset A$, \[\bigcap_{v \in A'}\st(\ii(\mathcal{H}), \{v\}) =  \st(\ii(\mathcal{H}), A').\]
We leave details to the reader.
\end{remark}
\begin{remark}\label{remark:gammaE}
Interestingly, we can give a lower bound on the connectivity of the noncover complex of a hypergraph in terms of the edge-wise domination number.

Let $\mathcal{H}$ be a hypergraph on $V$ with no isolated vertices.
For each $v \in V$, let $v^\vee = \{F \in \mathcal{H}: v \in F\}$.
Then the hypergraph $\mathcal{H}^\vee$ defined by 
\[V(\mathcal{H}^\vee) = \mathcal{H}\;\;\text{and}\;\;\mathcal{H}^\vee = \{v^\vee: v\in V\}\]
is called the {\em dual hypergraph} of $\mathcal{H}$.
In \cite[Corollary 4.15]{Tsu17}, it was shown that $\nc(\mathcal{H})$ and $\nc(\mathcal{H}^\vee)$ are homotopy equivalent.

Observe that $\mathcal{A} \in \ii(\mathcal{H}^\vee)$ if and only if $\bigcup_{F \in \mathcal{H} - \mathcal{A}} F = V$.
Hence for every $\mathcal{A} \in \ii(\mathcal{H}^\vee)$, we have $\gamma_E(\mathcal{H}) \leq |\mathcal{H}| - |\mathcal{A}|$.
Therefore, $\ii(\mathcal{H}^\vee)$ is $(|\mathcal{H}| - \gamma_E(\mathcal{H}))$-Leray.
It follows from Theorem~\ref{duality theorem} that $\eta(\nc(\mathcal{H})) = \eta(\nc(\mathcal{H}^\vee)) \geq \gamma_E(\mathcal{H}) - 1$.
\end{remark}

\medskip

\section{Tight paths and tight cycles}\label{tight hypergraphs}
The exact sequence \eqref{exact noncover} can be applied when computing the homology of the noncover complexes or independence complexes of hypergraphs.
In this section, we present two examples that are generalizations of cycles and paths.

Let $n$ and $k$ be positive integers and $V$ be a set of size $n$.
A $k$-uniform hypergraph on $V = \{v_1,\ldots,v_n\}$ is called the {\em $k$-uniform tight path}, denoted by $P_{n,k}$, if there exists a linear ordering $<$, say $v_1 < v_2 < \cdots < v_n$, on $V$ such that \[P_{n,k} := \{\{v_{i+1}, \ldots, v_{i+k}\}: 0 \leq i \leq n-k\}.\]
When $n < k$, then there is no edge.

For $n \geq k+1$, the {\em $k$-uniform tight cycle $C_{n,k}$} is the $k$-uniform hypergraph on $\mathbb{Z}_n$ such that 
\[C_{n,k} := \{\{{i}, \ldots, {i+k-1}\}: 0 \leq i \leq n-1\}.\]
For example, $P_{n,2}$ and $C_{n,2}$ are a path and a cycle, respectively.
See Figure~\ref{fig:tight} for an illustration of the $3$-uniform tight path and $3$-uniform tight cycle on $6$ vertices.
\begin{figure}[htbp]
    \centering
    \includegraphics[scale=.8]{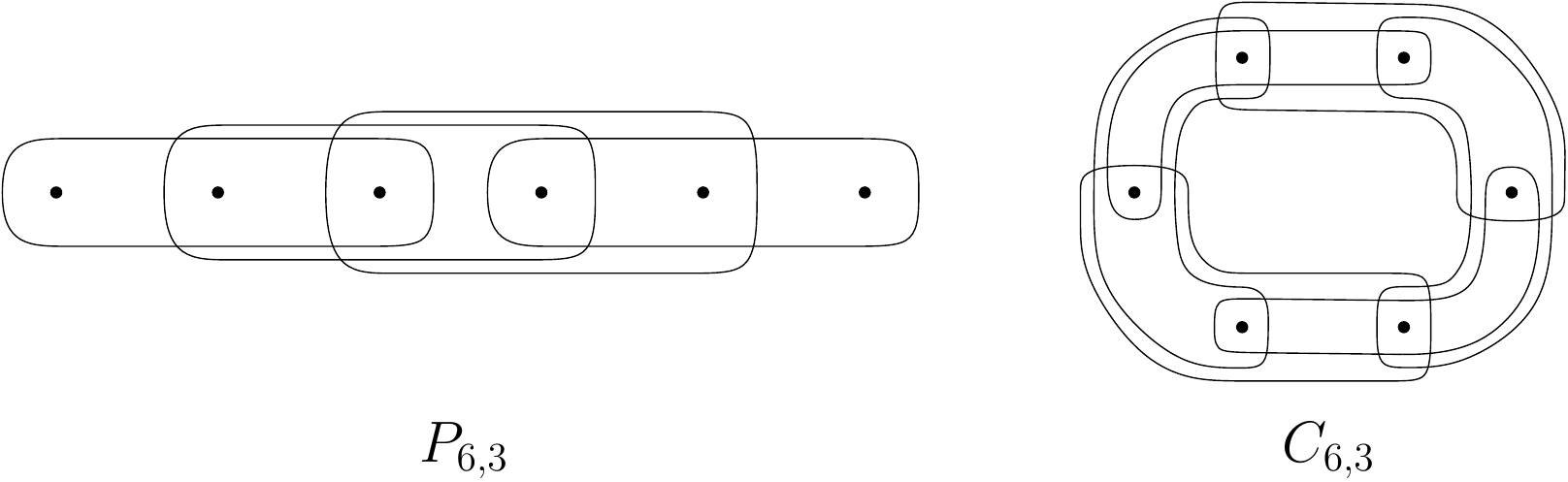}
    \caption{$3$-uniform tight path and tight cycle on $6$ vertices.}
    \label{fig:tight}
\end{figure}

In \cite{Koz99} (see also \cite[Claim 3.3]{Mes03}), it was shown that for every integer $i \geq 0$,
\begin{align}\label{path-cycle}
\begin{split}
&\tilde{\beta}_i(\ii(P_{n,2})) = \begin{cases}
1 & \text{ if } n=3i+2, 3i+3,\\
0 & \text{otherwise}
\end{cases},\text{ and }\\
&\tilde{\beta}_i(\ii(C_{n,2})) = \begin{cases}
2 & \text{if }n=3i+3,\\
1 & \text{if }n=3i+2, 3i+4,\\
0 & \text{otherwise}.
\end{cases}
\end{split}
\end{align}
Applying the sequence \eqref{exact noncover} repeatedly, we can compute the reduced Betti numbers for independence complexes of $P_{n,k}$ and $C_{n,k}$ as a generalization of \eqref{path-cycle}.
For a positive integer $n$, we denote by $[n]$ the set $\{1,\ldots,n\}$.

\begin{theorem}\label{thm:tight}
Let $k, n$ be positive integers and let $q$ be a non-negative integer.
Then
\[\tilde{\beta}_i(\ii(P_{n,k})) = \begin{cases}
1 & \text{ if } i=q(k-1)+k-2, n=q(k+1)+k \\
~&\text{ or }i=q(k-1)+k-2, n=(q+1)(k+1),\\
0 & \text{otherwise}.
\end{cases}.\]
If $n > k$, then 
\[
\tilde{\beta}_i(\ii(C_{n,k}))=
\begin{cases}
k & \text{ if } i=q(k-1)+k-2, n=(q+1)(k+1),\\
1 & \text{ if } i=q(k-1)+k+t-3,\\ & \;\;\;n=(q+1)(k+1)+t \text{ for } t\in[k],\\
0 & \text{otherwise}.
\end{cases}
\]
\end{theorem}

Indeed, the results in \cite{Koz99} is stronger than \eqref{path-cycle}: the author computed the homotopy types of $\nc(P_{n,2})$ and $\nc(C_{n,2})$, respectively.
In this section, we prove hypergraph analogues of Kozlov's results.
Before we go into details, here are some topological backgrounds.

Let $X$ and $Y$ be topological spaces.
\begin{itemize}
\item We write $X \simeq Y$ if $X$ and $Y$ are homotopy equivalent. In particular, if $X$ is contractible, we write $X \simeq *$.

\item The {\em wedge sum} of $X$ and $Y$ is the space $X \vee Y$ obtained by taking the disjoint union of $X$ and $Y$ and identifying a point of $X$ and a point of $Y$.

\item Given an equivalence relation $\sim$ on $X$, we denote by $X / \sim$ the quotient space of $X$ under $\sim$.
When $Y \subset X$, we denote by $X / Y$ the quotient space $X / \sim$ where for all $x \neq y$ in $X$, $x \sim y$ if and only if $x, y \in Y$.

\item The {\it suspension} of $X$ is the quotient space $\Sigma X = X \times [0,1] / \sim$ where for all $(x, s) \neq (y, t)$ in $X \times [0,1]$, $(x, s) \sim (y, t)$ if and only if either $s=t = 0$ or $s=t = 1$.
Note that if $S^n$ is the $n$-dimensional sphere, then $\Sigma S^n \simeq S^{n+1}$.
\end{itemize}

Let $K$, $K_1$ and $K_2$ be simplicial complexes where $K_1 \cap K_2$ is not a void complex, and let $L$ be a subcomplex of $K$ that is not a void complex.
Then,
\begin{enumerate}[(A)]
\item\label{aaaaaa} If $K=K_1 \cup K_2$, then $K/K_2$ is homeomorphic to $K_1/(K_1 \cap K_2)$.
\item\label{bbbbbb} If the inclusion map $L \hookrightarrow K$ is homotopic to a constant map $c:L \to K$, i.e. if $L$ is contractible in $K$, then $K/L \simeq K \vee \Sigma L$. 
In particular, if $K \simeq *$, then $K/L \simeq \Sigma L$. If $L \simeq *$, then $K/L \simeq K$.
\end{enumerate}
Applying \eqref{aaaaaa} to Lemma~\ref{noncov union} and Lemma~\ref{noncov intersection}, we obtain the following.
\begin{lemma}\label{mv homotopy}
Let $\mathcal{H}$ be a hypergraph having at least two edges. Then for every inclusion-minimal edge $e$, $\nc(\mathcal{H}) \simeq \nc(\mathcal{H}-e) /\nc(\mathcal{H} \neg ~e)$.
\end{lemma}

\subsection{Tight paths}\label{sec:tight paths}
We will compute the homotopy types of the noncover complexes of hypergraphs that are more general than tight paths.
Let $j, k$ and $n$ be positive integers such that $j \leq k$.
We define $P_{n,k}^{(j)}$ as the hypergraph on $V = \{v_1,\ldots,v_n\}$ obtained from $P_{n,k} = \{\{v_{i+1}, \ldots, v_{i+k}\}: 0 \leq i \leq n-k\}$ by deleting the edge $\{v_{n-k+1},\ldots,v_n\}$ if it exists, and then adding the edge $\{v_{n-j+1},\ldots,v_{n}\}$ if $n \geq j$.
Note that $P_{n,k}^{(k)}=P_{n,k}$.
\begin{theorem}\label{path-homotopy}
Let $j, k$ and $n$ be positive integers such that $1 \leq j \leq k$.
If $n <j$, $\nc(P_{n,k}^{(j)})=\emptyset$.
If $n \geq j$, then for $q \geq 0$,
\[\nc(P_{n,k}^{(j)}) \simeq \begin{cases}
S^{2q} & \text{ if } n=(q+1)(k+1) \\
S^{2q-1} &\text{ if } n=q(k+1)+j,\\
$*$ & \text{otherwise}.
\end{cases}\]
\end{theorem}
\begin{proof}
For convenience, let $V = [n]$.

Let us first compute the homotopy type of $\nc(P_{n,k}^{(j)})$ for $1 \leq n \leq k+1$.
If $n<j$, then  $P_{n,k}^{(j)}$ has no edge and $\nc(P_{n,k}^{(j)})$ is a void complex.
If $j < n < k+1$, then $P_{n,k}^{(j)}$ has at least one isolated vertex since the only edge in $P_{n,k}^{(j)}$ has size $j < n$, and hence $\nc(P_{n,k}^{(j)})$ is contractible.
If $n=j$, then the only edge in $P_{j,k}^{(j)}$ is the whole vertex set and $\nc(P_{j,k}^{(j)})$ is an empty complex.
If $n=k+1$, then $\nc(P_{k+1,k}^{(j)})$ consists of the disjoint union of the simplex on $[k-j+1]$ and the vertex $k+1$.
Combining the above arguments, we obtain the following: for $1 \leq n \leq k+1$,
\begin{align}\label{small path}
\nc(P_{n,k}^{(j)}) \simeq \begin{cases}
\emptyset & \text{ if } n<j,\\
\{\emptyset\} = S^{-1} & \text{ if } n=j,\\
* & \text{ if } j < n \leq k,\\
S^{0} & \text{ if } n=k+1.\end{cases}
\end{align}

Now assume $n > k+1$.
Let $e$ be the the unique edge containing the vertex $n$.
Since $e$ is inclusion-minimal and $P_{n,k}^{(j)} -e \neq \emptyset$, Lemma~\ref{mv homotopy} implies 
\[\nc(P_{n,k}^{(j)}) \simeq \nc(P_{n,k}^{(j)}-e) / \nc(P_{n,k}^{(j)} \neg~e).\]
Since $n$ is an isolated vertex in $P_{n,k}^{(j)}-e$, $\nc(P_{n,k}^{(j)})$ is a cone with apex $n$, which is contractible.
Thus $\nc(P_{n,k}^{(j)}) \simeq \Sigma \nc(P_{n,k}^{(j)} \neg~e)$ by \eqref{bbbbbb}.
Note that the hypergraph $P_{n,k}^{(j)}\neg~e$ is the union of $P_{n-j,k}$ and edges \[\{\{t,t+1,\dots,n-j\}: n-k-j+2 \leq t \leq n-k\}.\]

By Observation~\ref{noncover inclusion-minimal}, we have $\nc(P_{n,k}^{(j)}\neg~e)=\nc(P_{n-j,k}^{(k-j+1)})$ since $P_{n-j,k}^{(k-j+1)}$ can be obtained from $P_{n,k}^{(j)}\neg~e$ by deleting all the edges that are not inclusion-minimal.
Thus we have 
\begin{align}\label{path-induction1}
\nc(P_{n,k}^{(j)}) \simeq \Sigma \nc(P_{n-j,k}^{(k-j+1)}).
\end{align}
If $n >k+j+1$, then by applying \eqref{path-induction1} to $\nc(P_{n-j,k}^{(k-j+1)})$, we have
\begin{align}\label{path-induction2}
\nc(P_{n,k}^{(j)}) \simeq \Sigma \nc(P_{n-j,k}^{(k-j+1)}) \simeq \Sigma^2 \nc(P_{n-k-1,k}^{(j)}).
\end{align}
Combining \eqref{small path}, \eqref{path-induction1} and \eqref{path-induction2}, we obtain the desired result.
\end{proof}

\subsection{Tight cycles}\label{sec:tight cycles}
Similarly, we will compute the homotopy types of the noncover complexes of hypergraphs that are more general than tight cycles.
Let $k$ and $n$ be positive integers such that $n > k$, and let $j$ be a non-negative integer.
Define $C_{n,k}^{(j)}$ as the hypergraph obtained from $C_{n,k}$ by deleting $j$ edges $\{\{s,\dots,s+k-1\}: 0 \leq s \leq j-1\}$.
Note that $C_{n,k}^{(0)}=C_{n,k}$ and $C_{n,k}^{(k-1)}=P_{n,k}$.

\begin{theorem}\label{cycle-homotopy}
Let $k$ and $n$ be positive integers with $n > k$ and $j \leq n$ a non-negative integer.
If $j = n$, then $\nc(C_{n,k}^{(j)})$ is a void complex.
If $k \leq j < n$, then $\nc(C_{n,k}^{(j)})$ is contractible.
If $0 \leq j \leq k-1$, then for $q \geq 0$,
\[
\nc(C_{n,k}^{(j)})=
\begin{cases}
\displaystyle\bigvee_{k-j} S^{2q} & \text{ if } n=(q+1)(k+1),\\
S^{2q+1} & \text{ if } n=(q+1)(k+1)+t, j <t \leq k,\\
* & \text{otherwise}.
\end{cases}
\]
\end{theorem}
\begin{proof}
It is clear that $\nc(C_{n,k}^{(n)})$ is the void complex since $C_{n,k}^{(n)} = \emptyset$.
If $k \leq j < n$, then the vertex $k-1$ becomes an isolated vertex and $C_{n,k}^{(j)} \neq \emptyset$, thus $\nc(C_{n,k}^{(j)})$ is contractible.

When $j = k-1$, since $C_{n,k}^{(k-1)}=P_{n,k}$, the statement is true by Theorem~\ref{path-homotopy}.
When $n = k+1$, it is obvious because $\nc(C_{n,k}^{(j)})$ is the disjoint union of $k + 1 - j$ vertices. 
Now assume $n > k+1$ and suppose that the statement holds for some $l \leq k-1$:
\[
\nc(C_{n,k}^{(l)})=
\begin{cases}
\displaystyle\bigvee_{k-l} S^{2q} & \text{ if } n=(q+1)(k+1),\\
S^{2q+1} & \text{ if } n=(q+1)(k+1)+t, l <t \leq k,\\
* & \text{otherwise}.
\end{cases}
\]

Let $e = \{l-1,\dots,k+l-2\} \in C_{n,k}^{(l-1)}$.
Since $C_{n,k}^{(l-1)}-e=C_{n,k}^{(l)}$, Lemma~\ref{mv homotopy} give us 
\[\nc(C_{n,k}^{(l-1)}) \simeq \nc(C_{n,k}^{(l)})/\nc(C_{n,k}^{(l-1)} \neg~e).\]
Thus, to determine the homotopy type of $\nc(C_{n,k}^{(l-1)})$, it is enough to compute the homotopy type of $\nc(C_{n,k}^{(l-1)} \neg~e)$.

Let $\mathcal{H}$ be the disjoint union of a loop $\{k+l-1\}$ and $P_{n-k-1,k}^{(l)}$ defined on $\{0,\ldots,l-2,k+l,\ldots,n-1\}$ with a linear ordering \[k+l \prec k+l+1 \prec \cdots \prec n-1 \prec 0 \prec 1 \prec \cdots \prec l-2.\]
Note that $\mathcal{H}$ can be obtained from $C_{n,k}^{(l-1)}\neg~e$ by deleting all the edges that are not inclusion-minimal, and hence $\nc(C_{n,k}^{(l-1)}\neg~e) = \nc(\mathcal{H})$ by Observation~\ref{noncover inclusion-minimal}.
Since $\mathcal{H} - \{k+l-1\}$ has an isolated vertex $k+l-1$, its noncover complex is contractible.
Observing that $\mathcal{H} \neg~\{k+l-1\} = P_{n-k-1,k}^{(l)}$, Lemma~\ref{mv homotopy} and \eqref{bbbbbb} implies 
\[\nc(\mathcal{H}) \simeq \nc(\mathcal{H}-\{k+l-1\})/\nc(\mathcal{H} \neg \{k+l-1\}) \simeq \Sigma \nc(P_{n-k-1,k}^{(l)}).\]
Thus we have
\[
\nc(C_{n,k}^{(l-1)}) \simeq \nc(C_{n,k}^{(l)})/ \nc(\mathcal{H}) \text{ and } \nc(\mathcal{H}) \simeq \Sigma \nc(P_{n-k-1,k}^{(l)}).
\]

If $n \not\equiv 0,l$ (mod $k+1$), then Theorem~\ref{path-homotopy} yields $\Sigma \nc(P_{n-k-1,k}^{(l)}) \simeq *$, and hence $\nc(\mathcal{H}) \simeq *$.
This implies $\nc(C_{n,k}^{(l-1)}) \simeq \nc(C_{n,k}^{(l)})$ by \eqref{bbbbbb}.

If $n=(q+1)(k+1)+l$ for some non-negative integer $q$, then by Theorem \ref{path-homotopy} and the induction hypothesis, we have
\[
\nc(\mathcal{H}) \simeq \Sigma \nc(P_{n-k-1,k}^{(l)}) \simeq \Sigma S^{2q-1} \simeq S^{2q} \text{ and } \nc(C_{n,k}^{(l)}) \simeq *.
\]
Thus, by \eqref{bbbbbb}, we obtain
\[
\nc(C_{n,k}^{(l-1)}) \simeq \Sigma \nc(\mathcal{H}) \simeq \Sigma S^{2q} \simeq S^{2q+1}.
\]

If $n=q(k+1)+k+1$ for some positive integer $q$, then by Theorem \ref{path-homotopy} and the induction hypothesis, we have
\[
\nc(\mathcal{H}) \simeq \Sigma \nc(P_{n-k-1,k}^{(l)}) \simeq \Sigma S^{2q-2} \simeq S^{2q-1} \text{ and } \nc(C_{n,k}^{(l)}) \simeq \displaystyle\bigvee_{k-l}S^{2q}.
\]
Since any map $S^{2q-1} \to \displaystyle\bigvee_{k-l}S^{2q}$ is homotopic to a constant map by the cellular approximation theorem, by \eqref{bbbbbb}, we obtain
\[
\nc(C_{n,k}^{(l-1)}) \simeq \displaystyle\bigvee_{k-l}S^{2q} \vee \Sigma S^{2q-1} =\displaystyle\bigvee_{k-l+1}S^{2q}.
\] 
\end{proof}

\medskip

\section{Leray numbers of noncover complexes}\label{proof main 2}
In this section, we prove Theorem~\ref{leray numbers} which claims a stronger statement for some special cases.
Each part appears in a separate subsection.

\subsection{Strong total domination numbers}
Let $K$ be an abstract simplicial complex on $V$, and let $A$ and $B$ be subcomplexes of $K$ such that $K = A \cup B$.
If we take a subset $W$ of $V$, then the induced subcomplexes satisfies $K[W] = A[W] \cup B[W]$ and $(A \cap B)[W] = A[W] \cap B[W]$.
This implies the following.
\begin{proposition}\label{MV induced}
Let $K, A$ and $B$ be simplicial complexes such that $K = A \cup B$.
Then \[L(K) \leq \max\{L(A), L(B), L(A\cap B)+1\}.\]
\begin{proof}
Let $n = \max\{L(A), L(B), L(A\cap B)+1\}$.
For any $W \subset V(K)$, we know that $\tilde{H}_i(A[W]), \tilde{H}_i(B[W])$, $\tilde{H}_{i-1}((A \cap B)[W]) = 0$ for all $i \geq n$.
By applying the Mayer-Vietoris sequence for $K[W] = A[W] \cup B[W]$ and $(A \cap B)[W] = A[W] \cap B[W]$, we obtain that $\tilde{H}_i(K[W]) = 0$ for all $i \geq n$.
This shows $L(K) \leq n$.
\end{proof}
\end{proposition}

The proof of the first part of Theorem~\ref{leray numbers} is based on Lemma~\ref{noncov union}, Lemma~\ref{noncov intersection} and Proposition~\ref{MV induced}.

	\begin{theorem}\label{gamma tilde leray}
	Let $\mathcal{H}$ be a hypergraph on $V$. 
	Suppose $\mathcal{H}$ contains no isolated vertices and every $e \in \mathcal{H}$ has size $|e| \leq 3$.
	Then the noncover complex $\nc(\mathcal{H})$ is $(|V| - \left\lceil\frac{\tilde{\gamma}(\mathcal{H})}{2}\right\rceil -1)$-Leray.
	\end{theorem}
	\begin{proof}
	We apply induction on $|\mathcal{H}| + |V|$.
	If $\mathcal{H}$ has no edge, then the statement is true since $\nc(\mathcal{H})$ is a void complex.
	If $\mathcal{H} = \{V\}$, then the statement is true since $\nc(\mathcal{H})$ is an empty complex.
	If $\binom{V}{1} \subset \mathcal{H}$, then the statement is true since $\tilde{\gamma}(\mathcal{H}) = 0$ and $\nc(\mathcal{H})$ is the boundary of the simplex on $V$ which is $(|V|-1)$-Leray.
    Note that any proper induced subcomplex is a simplex.
	
	\smallskip
	
	Take an edge $e$ with $|e| \geq 2$.
	If $e$ is not inclusion-minimal, then $\tilde{\gamma}(\mathcal{H}) \leq \tilde{\gamma}(\mathcal{H} - e) < \infty$ and $\nc(\mathcal{H}) = \nc(\mathcal{H} - e)$ by Observation~\ref{noncover inclusion-minimal}.
	Hence the statement is true by the induction hypothesis.
	If $e$ is inclusion-minimal, then we apply the exact sequence \eqref{exact noncover}.
	By Lemma~\ref{inequalities}, we have $\tilde{\gamma}(\mathcal{H}\neg~e) \geq \tilde{\gamma}(\mathcal{H}) - 2|e| + 2$, thus $\nc(\mathcal{H} \neg~e)$ is $(|V|-\left\lceil\frac{\tilde{\gamma}(\mathcal{H})}{2}\right\rceil-2)$-Leray by the induction hypothesis. 
	By Proposition~\ref{MV induced}, it is sufficient to show that $\nc(\mathcal{H} -e)$ is $(|V| - \left\lceil\frac{\tilde{\gamma}(\mathcal{H})}{2}\right\rceil - 1)$-Leray.

	Assume that $\mathcal{H} - e$ has no isolated vertices.
	Then we obviously have $\tilde{\gamma}(\mathcal{H}) \leq \tilde{\gamma}(\mathcal{H}-e) < \infty$.
	By the induction hypothesis, $\nc(\mathcal{H} - e)$ is $(|V| - \left\lceil\frac{\tilde{\gamma}(\mathcal{H})}{2}\right\rceil - 1)$-Leray.
	Thus we may assume that deleting the edge $e$ from $\mathcal{H}$ produces $k$ isolated vertices for some positive integer $k$.
	Let $\mathcal{H}'$ be the hypergraph obtained from $\mathcal{H} - e$ by deleting all isolated vertices.
	We claim that if $|e| \leq 3$, then $\tilde{\gamma}(\mathcal{H}') + 2k \geq \tilde{\gamma}(\mathcal{H})$.
	Then by induction, since $|V(\mathcal{H}')| = |V| - k$, we obtain that $\nc(\mathcal{H}')$ is $(|V| - \tilde{\gamma}(\mathcal{H}) - 1)$-Leray.
	
	Let $S \subset V(\mathcal{H}')$ be a minimum set which dominates $V(\mathcal{H}')$, i.e. $|S| = \tilde{\gamma}(\mathcal{H}')$. 
	If $k = 1$, let $v \in e$ be the only isolated vertex in $\mathcal{H} - e$.
	Then $S \cup (e\setminus\{v\})$ dominates $\mathcal{H}$, thus $\tilde{\gamma}(\mathcal{H}') + |e| - 1 \geq \tilde{\gamma}(\mathcal{H})$.
	Since $|e| \leq 3$, we obtain
	\[\tilde{\gamma}(\mathcal{H}') + 2 \geq \tilde{\gamma}(\mathcal{H}') + |e| - 1 \geq \tilde{\gamma}(\mathcal{H}).\]
	If $k \geq 2$, we have $\tilde{\gamma}(\mathcal{H}') + |e| \geq \tilde{\gamma}(\mathcal{H})$ since $S \cup e$ dominates $\mathcal{H}$.
	Since $|e| \leq 3 < 2k$, the claim is true because
	\[\tilde{\gamma}(\mathcal{H}') + 2k > \tilde{\gamma}(\mathcal{H}') + |e| \geq \tilde{\gamma}(\mathcal{H}).\]
	This completes the proof.
	\end{proof}
	
	The following example shows Theorem~\ref{gamma tilde leray} is the best possible in the sense that the condition $|e| \leq 3$ for every edge cannot be improved.
	\begin{example}\label{gamma tilde example}
	Let $\mathcal{H}_r$ be a hypergraph on $V = \{v_1,\ldots,v_{2r+1}\}$, whose edges are
	\begin{align*}
	\mathcal{H}_r = \{&\{v_1,\ldots,v_r\},\{v_2,v_{r+1}\},\{v_3,v_{r+1}\},\ldots,\{v_r,v_{r+1}\},
	\\&\{v_{r+1},v_{r+2}\},\{v_{r+1},v_{r+3}\},\ldots,\{v_{r+1},v_{2r}\},\{v_{r+2},\ldots,v_{2r+1}\}\}.
	\end{align*}
	We will show that if $r \geq 4$, then $\nc(\mathcal{H}_r)$ is not $(|V| - \left\lceil\frac{\tilde{\gamma}(\mathcal{H}_r)}{2}\right\rceil -1)$-Leray.
	
	We first claim that $\tilde{\gamma}(\mathcal{H}_r) = 2r-1$.
	Since $\{v_1,\ldots,v_r\}$ is the only edge which contains the vertex $v_1$, we need $\{v_2,\ldots,v_r\}$ to dominate $v_1$.
	Similarly, we need $\{v_{r+2},\ldots,v_{2r}\}$ to dominate $v_{2r+1}$.
	Thus every set $S \subset V$ which dominates $V$ should contain the set $W = \{v_2,\ldots,v_r,v_{r+2},\ldots,v_{2r}\}$ of $2r-2$ vertices.
	However, $S \setminus W$ should not be empty since $W$ cannot dominate any vertex in $W$.
	This implies $|S| \geq 2r-1$.
	Indeed, since the set $W \cup \{v_{r+1}\}$ dominates $V$, we have $\tilde{\gamma}(\mathcal{H}_r) = 2r-1$ and $|V| - \left\lceil\frac{\tilde{\gamma}(\mathcal{H}_r)}{2}\right\rceil -1  = r$.
	On the other hand, the induced subcomplex $\nc(\mathcal{H})[W]$ is the boundary of the simplex on $W$.
	Hence we have $H_{2r-4}(\nc(\mathcal{H}_r)[W]) \neq 0$, showing that $\nc(\mathcal{H}_r)$ is not $r$-Leray whenever $r \geq 4$.
	See Figure~\ref{fig:gammatilde} for the illustration when $r = 4$.	
	\begin{figure}[htbp]
    \centering
    \includegraphics[scale=1.05]{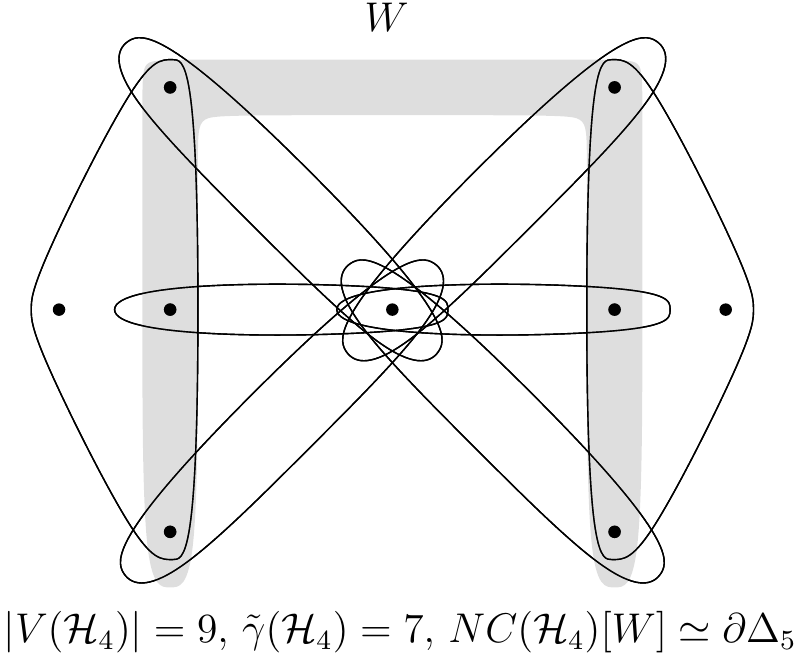}
    \caption{$|V(\mathcal{H}_4)| - \left\lceil\frac{\tilde{\gamma}(\mathcal{H}_4)}{2}\right\rceil -1 = 4$ but $\nc(\mathcal{H}_4)$ is not $4$-Leray.}
    \label{fig:gammatilde}
\end{figure}
	\end{example}

\subsection{Strong independence domination numbers}
	In this section, we prove the second part of Theorem~\ref{leray numbers}.
	When $\mathcal{H}$ is a graph, it is implied by the main result in \cite{CKP19}.
	Our result gives an alternative proof with a slight extension.
	
	For a hypergraph $\mathcal{H}$ on $V$ and $v \in V$, let \[\mathcal{H} \neg~v := \{e\setminus\{v\}: e\in\mathcal{H}\}.\]
	Clearly, if $\mathcal{H}$ has no isolated vertices, then $\mathcal{H} \neg~v$ also has no isolated vertices.
	Note also that $\nc(\mathcal{H} \neg~v) = \nc(\mathcal{H})[V\setminus\{v\}]$.
	
	\begin{theorem}\label{gamma si leray}
	Let $\mathcal{H}$ be a hypergraph on $V$. 
	Suppose $\mathcal{H}$ contains no isolated vertices and every $e \in \mathcal{H}$ has size $|e| \leq 2$.
	Then the noncover complex $\nc(\mathcal{H})$ is $(|V| - \gamma_{si}(\mathcal{H}) -1)$-Leray.	\end{theorem}
	\begin{proof}
	We apply induction on $|V|$.
	If $|V| = 1$, then the statement is obviously true.
	Thus we may assume that $|V| > 1$.
	It is sufficient to show the inequality $\gamma_{si}(\mathcal{H} \neg~v) \geq \gamma_{si}(\mathcal{H}) - 1$ for every $v \in V$.
	Then we have \[|V(\mathcal{H}\neg~v)|-\gamma_{si}(\mathcal{H}\neg~v)-1 \leq |V|-\gamma_{si}(\mathcal{H})-1,\] and it follows from the induction hypothesis that $\nc(\mathcal{H})[V\setminus\{v\}]$ is $(|V(\mathcal{H})|-\gamma_{si}(\mathcal{H})-1)$-Leray.
	Let $A$ be a strongly independent set of $\mathcal{H}$ such that $\gamma(\mathcal{H}; A) = \gamma_{si}(\mathcal{H})$.
	We will show that for every vertex $v$ there exists a set $S$ in $\mathcal{H}\neg~v$ such that $|S| \leq \gamma_{si}(\mathcal{H}\neg~v)$ and $|S|+1 \geq \gamma_{si}(\mathcal{H})$.
	
	If $v \notin A$, let $A'$ be the set of all vertices in $A$ which are not dominated by $v$.
	Obviously $A'$ is a strongly independent set of $\mathcal{H} \neg~v$.
	Let $S$ be a smallest set which dominates $A'$ in $\mathcal{H}\neg~v$.
	Then $S \cup \{v\}$ dominates $A$ in $\mathcal{H}$.
	If $v \in A$, we take a vertex $u \notin A$ such that $\{u,v\} \in \mathcal{H}$.
	The existence of $u$ is guaranteed because $\mathcal{H}$ has no isolated vertices.
	Let $S$ be a smallest set which dominates $A\setminus\{v\}$ in $\mathcal{H}\neg~v$.
	Then $S \cup \{u\}$ dominates $A$ in $\mathcal{H}$.
	In either case, we have $|S| \leq \gamma_{si}(\mathcal{H}\neg~v)$ and $|S|+1 \geq \gamma_{si}(\mathcal{H})$, as desired.
	\end{proof}

	The following example shows that the condition $|e| \leq 2$ for every edge $e \in \mathcal{H}$ in Theorem~\ref{gamma si leray} cannot be improved.
		\begin{example}\label{gammasi example}
	For $r \geq 3$, consider an $r$-uniform hypergraph
	\[\mathcal{F}_r := \{\{(i,1),\ldots,(i,r))\}: i \in [r]\} \cup \{\{(1,i),\ldots,(r,i)\}: i \in [r]\setminus\{1\}\}\]
	defined on $[r]\times[r]$.
	Clearly $A = \{(i,1): i \in [r]\}$ is a strongly independent set of $\mathcal{F}_r$.
	Since $\{(i,1),\ldots,(i,r))\}$ is the only edge in $\mathcal{H}_r$ which contains the vertex $(i,1)$ for each $i$, the set \[\{(i, j): i \in [r]\;\text{and}\;j \in [r]\setminus\{1\}\}\] is the only set which dominates $A$.
	This shows that $\gamma_{si}(\mathcal{F}_r) \geq (r-1)r$.
	
	On the other hand, we will show that $\nc(\mathcal{F}_r)$ is not $(r-1)$-Leray whenever $r \geq 3$.
	Note that if $r \geq 3$ then $|V(\mathcal{F}_r)| - \gamma_{si}(\mathcal{F}_r) - 1 \leq r-1 \leq 2r-4$.
	Let \[W = \{(i,1): i \in [r]\setminus\{r\}\}\cup\{(r,i): i \in [r]\setminus\{1\}\}.\]
	It is clear that the induced subcomplex $\nc(\mathcal{F}_r)[W]$ is the boundary of the simplex on $W$, thus we have $H_{2r-4}(\nc(\mathcal{F}_r)[W]) \neq 0$.
	See Figure~\ref{fig:gammasi} for the illustration when $r = 4$.
\begin{figure}[htbp]
    \centering
    \includegraphics[scale=.66]{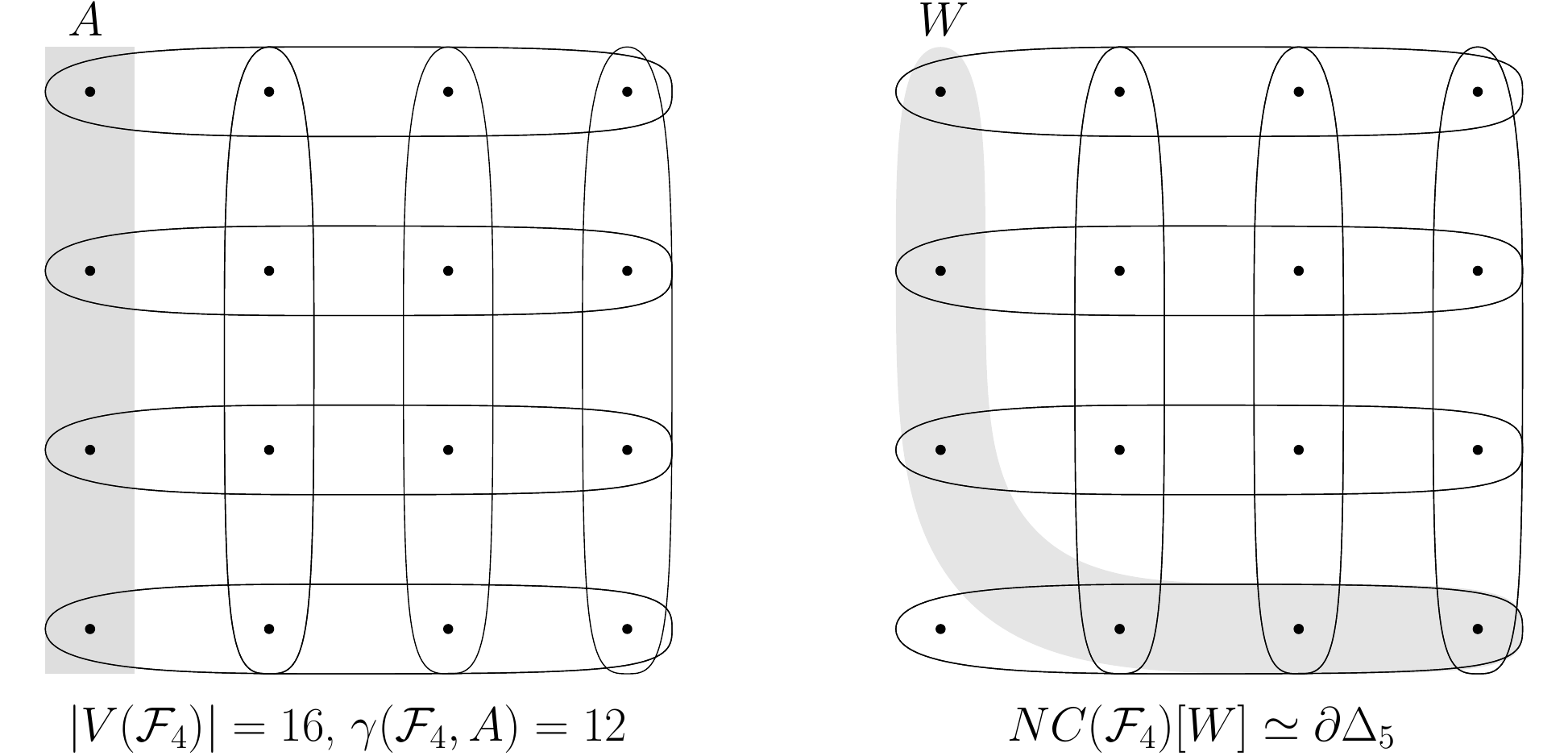}
    \caption{$|V(\mathcal{F}_4)| - \gamma_{si}(\mathcal{F}_4) -1 \leq 3$ but $\nc(\mathcal{F}_4)$ is not $4$-Leray.}
    \label{fig:gammasi}
\end{figure}
	\end{example}

\subsection{Edgewise-domination numbers}	
	For $\gamma_E(\mathcal{H})$, we can prove the following without any restriction on the size of edges in $\mathcal{H}$.
	\begin{theorem}\label{gamma e leray}
	Let $\mathcal{H}$ be a hypergraph on $V$ with no isolated vertices.
	Then the noncover complex $\nc(\mathcal{H})$ is $(|V| - \gamma_{E}(\mathcal{H}) -1)$-Leray.
	\end{theorem}
	\begin{proof}
	We apply induction on $|V|$.
	The statement is obvious when $|V| = 1$.
	Take an arbitrary vertex $v \in V$.	
	It is sufficient to show that $\gamma_E(\mathcal{H} \neg~v) \geq \gamma_E(\mathcal{H}) -1$.
	
	Let $\mathcal{F}$ be a subgraph of $\mathcal{H} \neg~v$ such that $\bigcup_{A\in\mathcal{F}}A$ dominates $V(\mathcal{H}\neg~v) = V \setminus \{v\}$ in $\mathcal{H}\neg~v$ and $|\mathcal{F}| = \gamma_E(\mathcal{H}\neg~v)$.
	For each $A \in \mathcal{F}$, let $A' = A$ if $A \in \mathcal{H}$ and let $A' = A \cup \{v\}$ if $A \notin \mathcal{H}$.
	Since $\mathcal{H}$ contains no isolated vertices, we can take an edge $A_v \in \mathcal{H}$ such that $v \in A_v$.
	Then it is obvious that $\bigcup_{A'\in\mathcal{F}'}A'$ dominates $V$ where $\mathcal{F}' = \{A': A\in\mathcal{F}\} \cup \{A_v\}$, implying $\gamma_E(\mathcal{H} \neg~v)+1 \geq \gamma_E(\mathcal{H})$.
	\end{proof}

\medskip
\section{Concluding remarks}\label{concluding remarks}
\subsection{Homological connectivity of general position complexes}\label{genpos}
In this section, we discuss an application of Corollary~\ref{eta} to the homological connectivity of ``general position complexes''.

Let $P$ be a set of points in $\mathbb{R}^d$, and let $G(P)$ denote the simplicial complex consisting of those subsets of $P$ which are in general position.
Furthermore, let $\varphi(P)$ denote the maximal size of a subset of $P$ in general position, that is, $\varphi(P) = \dim(G(P))+1$.
In \cite{HMM16}, it was shown that if $\varphi(P) > d\binom{2k-2}{d}$ then $\eta(G(P)) \geq k$.
We give an alternative proof of it, by showing the following matroidal generalization.
\begin{lemma}\label{lem:matroids}
Let $M$ be a matroid of rank $r$ on $X$.
For any finite subset $Y$ of $X$, define a hypergraph \[\mathcal{H}_Y=\{S\subseteq Y:|S|\leq r, S\text{ is a circuit of }M\}.\]
If $\mathcal{H}_Y$ has an independent set of size greater than $(r-1)\binom{2k-2}{r-1}$, then $\eta(\ii(\mathcal{H}_Y)) \geq k$.
\end{lemma}
\begin{proof}
By Corollary~\ref{eta}, it is sufficient to show that $\tilde{\gamma}(\mathcal{H}_Y)>2k-2$.
Let $W$ be a subset of $Y$ of size $2k-2$.
Take an independent set $A$ of $\mathcal{H}_Y$ with $|A|>(r-1)\binom{2k-2}{r-1}$.
Observe that any subset of $A$ of size $r$ is independent in the matroid $M$. 
We claim that $W$ cannot dominate $A$.

Any vertex $v$ can be dominated by $W$ in $\mathcal{H}_Y$ if and only if there is $U \subset W$ of size $r-1$ such that $v$ is contained in a submatroid $M_U$ of $M$ generated by $U$.
Each $U \subset W$ with $|U| = r-1$ dominates at most $r-1$ vertices of $A$ in $\mathcal{H}_Y$ since it generates a submatroid $M_U$ of rank at most $r-1$ of $M$.
Since $|\binom{W}{r-1}| = \binom{2k-2}{r-1}$, $W$ dominates at most $(r-1)\binom{2k-2}{r-1}$ vertices of $A$ in $\mathcal{H}_Y$.
This completes the proof.
\end{proof}

Let $g_d(k)$ be the minimum integer such that for any $P \subset \mathbb{R}^d$, if $\varphi(P) \geq g_d(k)$ then $G(P)$ is $k$-connected.
Since affine independence on $P$ defines a matroid on $P$, Lemma~\ref{lem:matroids} yields an upper bound on $g_d(k)$ which is in $O(k^d)$. Here we give an example which shows that this bound is asymptotically tight, in other words, we show that the function $g_d(k)$ is in $\Theta(k^d)$.

\begin{example}\label{ex:genpos}
Let $A$ be a set of $n > d$ points in $\mathbb{R}^d$ which are in general position. Let $H$ be the set of $N=\binom{n}{d}$ hyperplanes spanned by the $d$-tuples of points in $A$.
Let $B$ be a set of $N$ points in general position in $\mathbb{R}^d$ such that $|B\cap h|=1$ for every hyperplane $h\in H$. Let $P = A\cup B$. Notice that $|P| = N+n$ and $\varphi(P) = N+d-1$.

\begin{claim}\label{homologyGP}
$\tilde{H}_i(G(P)) = 0$ if and only if $i\neq n-1$. 
\end{claim}

Before we prove Claim \ref{homologyGP} we note that $G(P)$ is the independence complex of the $(d+1)$-uniform hypergraph $\mathcal{F}$ on $P$ where each edge of $\mathcal{F}$ corresponds to a $d$-tuple $S \subset A$ together with the corresponding point $x \in B$ lying in the hyperplane spanned by $S$. 
Since $B$ is strongly independent in $\mathcal{F}$ and we need $A$ to dominate $B$ in $\mathcal{F}$, we have $\gamma_{si}(\mathcal{F}) \geq n$.
By Corollary~\ref{eta} we get $\tilde{H}_i(G(P))= 0$ for all $i\leq n-2$.

\begin{proof}[Proof of Claim \ref{homologyGP}]
We apply the nerve theorem to determine $\tilde{H}_i(\nc(\mathcal{F}))$. 
The inclusion maximal simplices of $\nc(\mathcal{F})$ are formed by the complements of the edges of $\mathcal{H}_P$ which can be labeled by the $d$-tuples of $A$. 
Let $X_1, \dots, X_N$ denote these simplices. 
Note that each $X_i$ has dimension  $N+n-d-2$. 
Clearly we have $\nc(\mathcal{F}) = \bigcup_{j \in [N]} X_j$, $\bigcap_{j \neq i}X_j \neq \emptyset$ for each $i \in [N]$, and $\bigcap_{j \in [N]} X_j = \emptyset$.
Therefore, by the nerve theorem, $G(P)$ is homotopy equivalent to the boundary of the $(N-1)$-simplex which gives us $\tilde{H}_i(\nc(\mathcal{F}))= 0$ if and only if $i\neq N-2$. Since $|P| = N+n$, the claim now follows by Theorem~\ref{duality theorem}.
\end{proof}
\end{example}

\subsection{Rainbow covers of hypergraphs}\label{sec:rainbow cover}
As an application of Theorem~\ref{leray numbers}, we can obtain the following result for ``rainbow covers''.
Let $l$ and $m$ be positive integers with $l \leq m$.
Given $m$ covers $X_1,\ldots,X_m$ in a hypergraph $\mathcal{H}$, a {\em rainbow cover} of size $l$ is a cover $X = \{x_{i_1},\ldots,x_{i_l}\}$ of $l$ distinct vertices of $\mathcal{H}$ such that $1 \leq i_1 < \cdots < i_l \leq m$ and $x_{i_j} \in X_{i_j}$ for each $j \in \{1,\ldots,l\}$.

\begin{corollary}\label{rainbow covers}
Let $\mathcal{H}$ be a hypergraph with no isolated vertices.
Then each of the following holds:
\begin{enumerate}
\item Suppose that every edge in $\mathcal{H}$ has size at most $3$.
Then for every $|V(\mathcal{H})| - \left\lceil\frac{\tilde{\gamma}(\mathcal{H})}{2}\right\rceil$ covers of $\mathcal{H}$, there exists a rainbow cover.

\item Suppose that every edge in $\mathcal{H}$ has size at most $2$.
Then for every $|V(\mathcal{H})| - \gamma_{si}(\mathcal{H})$ covers of $\mathcal{H}$, there exists a rainbow cover.

\item For every $|V(\mathcal{H})| - \gamma_E(\mathcal{H})$ covers of $\mathcal{H}$, there exists a rainbow cover.
\end{enumerate}
\end{corollary}

Corollary~\ref{rainbow covers} follows from the topological colorful Helly theorem, which we state here as a special case of a famous result by Kalai and Meshulam~\cite{KM05}.
\begin{theorem}[Topological colorful Helly theorem]\label{km}
Let $K$ be a $d$-Leray simplicial complex with a vertex partition $V(K) = V_1 \cup \cdots \cup V_m$ with $m \geq d+1$.
If $\sigma \in K$ for every $\sigma \subset V(K)$ with $|\sigma \cap V_i| = 1$, then there exists $I \subset \{1, \ldots , m\}$ of size at least $m - d$ such that $\bigcup_{i \in I} V_i \in K$.
\end{theorem}

\subsection{Another independence domination parameter}\label{sec:weakdom}
Bounding Leray numbers of noncover complexes in terms of domination numbers of (hyper)graphs also has been studied from an algebraic viewpoint.
For example, see \cite{DS13, DS15, DS-15}.
See \cite{Hoch77} (or \cite[Theorem 1.3]{KM06}) to understand relations between algebra and topology of an abstract simplicial complex in this context.
Especially, it is worth to mention a result in \cite{DS-15} which deals with another independence domination parameter in hypergraphs.

Let $\mathcal{H}$ be a hypergraph on $V$.
Let \[\gamma'(\mathcal{H}; A) := \min\{|W|: W\subset V\setminus A,~W\;\text{weakly dominates}\;v\text{ for each }v\in A\},\]
and $t(\mathcal{H}) := \max\{\gamma'(\mathcal{H}; A):A \in \ii(\mathcal{H})\}$.
The following is a reformulation of Theorem 5.2 in \cite{DS-15}.
\begin{theorem}\label{weak}
Let $\mathcal{H}$ be a hypergraph on $V$ with no isolated vertices.
Then $L(\nc(\mathcal{H})) \leq |V| - t(\mathcal{H}) - 1$.
\end{theorem}
Consequently, we obtain $\eta(\ii(\mathcal{H})) \geq t(\mathcal{H})$.
Also, Theorem~\ref{weak} gives an analogue of Corollary~\ref{rainbow covers}: every $|V|-t(\mathcal{H})$ covers in $\mathcal{H}$ assigns a rainbow cover.
Note that $t(\mathcal{H}) = \gamma_{si}(\mathcal{H})$ when $\mathcal{H}$ is a graph.

The two independence domination parameters $t(\mathcal{H})$ and $\gamma_{si}(\mathcal{H})$ are not comparable in general.
In particular, we can make examples so that one of the parameters is arbitrarily large while the other remains as a constant.
\begin{example}
Let $\mathcal{H}$ be a complete $k$-uniform hypergraph $\binom{[n]}{k}$ on $n \geq k$ vertices. Then we have $\gamma_{si}(\mathcal{H}) = k-1$ and $t(\mathcal{H}) = 1$.
\end{example}
\begin{example}
Let $k$ and $n$ be positive integers such that $k \geq 3$ and $n \geq 2$.
We construct a $k$-uniform hypergraph $\mathcal{A}_{n,k}$ with vertex set $V_{n,k}$ such that $|V_{n,k}| = \binom{(k-1)n}{k-1} + (k-1)n$ as follows.

Let $W_{n,k} \subset V_{n,k}$ be a subset of size $(k-1)n$.
Consider a bijection $\phi: \binom{W_{n,k}}{k-1} \to V_{n,k} \setminus W_{n,k}$.
Now we define the edges of $\mathcal{A}_{n,k}$ as 
\[\mathcal{A}_{n,k} := \left\{\{\phi(X)\} \cup X: X \subset \binom{W_{n,k}}{k-1}\right\} \cup \binom{V_{n,k} \setminus W_{n,k}}{k}.\]

Since $W_{n,k}$ is an independent set and $\gamma'(\mathcal{A}_{n,k}; W_{n,k}) = n$, we have $t(\mathcal{A}_{n,k}) \geq n$.
Observe that any strongly independent set of $\mathcal{A}_{n,k}$ contains at most one vertex from $W_{n,k}$ and at most one vertex from $V_{n,k} \setminus W_{n,k}$.
Take $u \in W_{n,k}$ and $v \in V_{n,k} \setminus W_{n,k}$ such that $u$ and $v$ are not contained in the same edge of $\mathcal{A}_{n,k}$.
Then the strongly independent set $\{u,v\}$ can be dominated by $k$ vertices.
First observe that a $(k-1)$-set $\phi^{-1}(v)$ in $W_{n,k}$ dominates $v$.
Then take any $(k-2)$-subset $U$ in $\phi^{-1}(v)$ and let $u' = \phi(U \cup \{u\})$.
Clearly $U \cup \{u'\}$ dominates $u$.
This shows $\gamma_{si}(\mathcal{A}_{n,k}) \leq k$.
\end{example}

There can be other notions of domination in hypergraphs, and it is an interesting problem to figure out relations between each of the domination parameters and the topology of the noncover complexes of hypergraphs.

\section*{Acknowledgements}
This paper was conducted while the authors were post-doctoral fellows at the Technion -- Israel Institute of Technology.
A part of this research was done while Jinha Kim was a graduate student at Seoul National University.
Jinha Kim was supported by BSF grant No. 2016077 and ISF grant No. 1357/16, and Minki Kim was supported by ISF grant No. 936/16.
This work was supported by the Institute for Basic Science (IBS-R029-C1).

We are grateful to Ron Aharoni for his advice that led to improving the readability of the paper.
Andreas Holmsen made helpful comments and suggestions in the early stages of the project, especially in Section~\ref{genpos}.
We are also grateful to anonymous referees for insightful comments that helped improve the exposition.
In particular, we are indebted to the referee who pointed out that our initial results about the Betti numbers of noncover complexes in Section~\ref{tight hypergraphs} can be improved to results about the homotopy types.

\end{document}